\documentclass[reqno,11pt]{amsart}
\usepackage{latexsym,amssymb,amsfonts,amsmath,amsthm}
\usepackage{mathrsfs}
\usepackage[usenames]{color}
\newtheorem{thm}{Theorem}[section]

\newtheorem{lem}[thm]{Lemma}
\newtheorem{prop}[thm]{Proposition}
\theoremstyle{definition}

\theoremstyle{remark}
\newtheorem{rem}[thm]{Remark}
\numberwithin{equation}{section}
\newcommand{\Real}{\mathbb R}
\newcommand{\eps}{\varepsilon}
\newcommand{\diag}{\mathrm{diag}}

\newcommand{\F}{\mathscr{F}}

\newcommand{\s}{\mathbb{S}}

\newcommand{\one}[1]{\mathbf{1}_{\{#1\}}}

\renewcommand{\P}{\mathsf{P}}

\newcommand{\E}{\mathsf{E}}
\newcommand{\osimplex}{\mathcal{S}^{d-1}}
\newcommand{\csimplex}{\bar{\mathcal{S}}^{d-1}}

\newcommand{\argmax}{\mathrm{argmax}}

\newcommand{\M}{\mathcal{M}}

\newcommand{\K}{\mathbb{K}}

\begin{document}

\title[MLE for continuous time HMM]{Maximum Likelihood Estimator for Hidden Markov Models in continuous time}%
\author{Pavel Chigansky}%
\address{Department of Statistics,
The Hebrew University,
Mount Scopus, Jerusalem 91905
Israel }
\email{pchiga@mscc.huji.ac.il}

\thanks{This article has been written during the author's visit at
Laboratoire de Statistique et Processus, Universite du Maine, France,
supported by the Chateaubriand fellowship.}

%\subjclass{}%

\date{August, 25, 2007}%
%\dedicatory{}%
%\commby{}%
% ----------------------------------------------------------------
\begin{abstract}
The paper studies large sample asymptotic properties of the Maximum
Likelihood Estimator (MLE) for the parameter of a continuous time
Markov chain, observed in white noise. Using the method of weak
convergence of likelihoods due to I.Ibragimov and R.Khasminskii
\cite{IH}, consistency, asymptotic normality and convergence of
moments are established for MLE under certain strong ergodicity
conditions on the chain.
\end{abstract}
\maketitle

\section{Introduction}

\subsection{The setting and the main result}
Consider a pair of continuous time random processes
$(S,X)=(S_t,X_t)_{t\ge 0}$, where $S$ is a {\em signal} Markov chain
with values in a finite real set $\s=\{a_1,...,a_d\}$ and $X$ is
given by
$$
X_t = \int_0^t h(S_r)dr + B_t,
$$
with an  $\s \mapsto \Real$ function $h$ and a Brownian motion $B$,
independent of $S$. Let $\Lambda=(\lambda_{ij})$,
$i,j\in\{1,...,d\}$ and $\nu$ be the transition rates and the
initial distribution of the chain respectively. Suppose the model,
i.e. $\Lambda$ and $h$, depend on a parameter $\theta\in
\varTheta$ with $\varTheta$, being a bounded open subset of $\Real^n$,
which is to be estimated given the observed trajectory $X^T=\{X_s,
0\le s\le T\}$.

In this paper we study the large sample asymptotic properties of the
Maximum Likelihood Estimator (MLE) $\hat \theta_T$ of $\theta$ given
$X^T$. For a fixed value of the parameter, let $\P_\theta$ denote
the probability measure, induced by $(S,X)$ on the corresponding function space
$D_{[0,\infty)}\times C_{[0,\infty)}$, and let $\F^X_t$ be the natural
filtration of $X$. Introduce the {\em filtering} process
$\pi^\theta=(\pi^\theta_t)_{t\ge 0}$ with values in the simplex of
probability vectors $\csimplex=\{x\in \Real^d: x_i\ge 0,
\sum_{i=1}^d x_i=1 \}$, whose entries are the conditional
probabilities $\big\{\pi^\theta_t\big\}_{i}:=
\P_\theta\big(S_t=a_i|\F^X_t\big)$. As is well known, the process
$$
\bar B_t = X_t - \int_0^t h^* \pi^{\theta}_tdt
$$
is the innovation Brownian motion with respect to $\F^X_t$ and  by
the Girsanov theorem the likelihood, i.e. the Radon-Nikodym
derivative of  $P_\theta$, restricted to $\F^X_T$, with respect to
the Wiener measure on $C_{[0,T]}$, is given by
$$
L_T(\theta;X^T):= \exp\left\{\int_0^T h^*\pi^\theta_t dX_t -\frac 1 2
\int_0^T \big(h^*\pi^\theta_t\big)^2 dt\right\},
$$
where $h$ is the vector with entries $h_i=h(a_i)$, $i=1,...,d$ and
$h^*$ denotes its transposed.
We shall define the MLE $\hat \theta_T$ as a maximizer of the
likelihood:
\begin{equation}\label{mle1}
\hat \theta_T : = \argmax_{\theta\in \bar{\varTheta}}L_T(\theta;X^T)
\end{equation}
where $\bar\varTheta$ stands for the closure of $\varTheta$. If
$\Lambda$ and $h$ are continuous in $\theta$,
$L_T(\theta;X^T)$ is a continuous function of $\theta$  on
$\bar\varTheta$ with probability one  and hence the maximum value is
attained, perhaps at multiple values of $\theta$, in which case
any  maximizer is chosen.

In fact, for any
$\theta,\eta\in \varTheta$ the restrictions of $\P_\theta$ and
$\P_\eta$ on $\F^X_T$ are equivalent (see e.g. \cite{LSI}) with the
corresponding likelihood
\begin{multline*}
L_T(\theta,\eta;X^T):=\frac{d\P_\theta}{d\P_\eta}(X)_{\big|\F^X_T}= \\
\exp\left\{\int_0^T \big(h^*\pi^\theta_t-h^*\pi^\eta_t\big)dX_t
-\frac 1 2 \int_0^T
\Big[\big(h^*\pi^\theta_t\big)^2-\big(h^*\pi^\eta_t\big)^2\Big]
dt\right\}
\end{multline*}
and for any $\eta\in \varTheta$
\begin{equation}
\label{mle2} \hat \theta_T  = \argmax_{\theta\in
\bar{\varTheta}}L_T(\theta,\eta;X^T).
\end{equation}
The latter expression is more convenient for the analysis purposes and,
in fact, we shall  work with \eqref{mle2}, fixing
$\eta:=\theta_0$, where $\theta_0$ is the actual (unknown) value of
the parameter. This choice is quite natural as we study
$\hat\theta_T$ under measure $\P_{\theta_0}$.

To simplify the presentation, we shall consider the case of  scalar
parameter, i.e. $\varTheta \subset \Real$, and, moreover, assume
that $h$ does not depend on $\theta$ (this issue is briefly addressed
in Section \ref{sec-dis}). Our main result is the following theorem.
\begin{thm}\label{thm}
Assume
\begin{enumerate}
\renewcommand{\theenumi}{a-\arabic{enumi}}

\item\label{a-1} $\lambda_{ij}(\theta)$ are twice continuously differentiable on $\bar\varTheta$ and
\begin{equation}\label{serg}
\min_{i\ne j}\min_{\theta\in \bar\varTheta}\lambda_{ij}(\theta)>0;
\end{equation}

\item\label{a-2} the model is identifiable in the sense that the function
$g(\theta_0,\theta):=\E_{\theta_0}\big(h^*\check{\pi}^\theta_0-h^*\check{\pi}^{\theta_0}_0\big)^2$,
where $(\check\pi^\theta_0, \check\pi^{\theta_0}_0)$ are random
vectors, sampled from the unique invariant measure
\footnote{$(\pi^\theta_t, \pi^{\theta_0}_t)$ is indeed a Markov
process and it has a unique invariant measure under \eqref{a-1} -
see Lemma \ref{lem-inv} below} of the Markov process $(\pi^\theta_t,
\pi^{\theta_0}_t)$ under $\P_{\theta_0}$, satisfies
\begin{equation}
\label{ident} \inf_{\theta_0\in \K}\inf_{|\theta-\theta_0|\ge r
}g(\theta_0,\theta)>0, \quad \forall r>0
\end{equation}
for any compact $\K\subset \varTheta$;
\item\label{a-3} the Fisher information\footnote{
here $\dot \pi^\theta_t := \frac{\partial }{\partial
\theta}\pi^\theta_t$ in the $\P_{\theta_0}$-a.s. sense: such
derivative exists, when $\theta\mapsto \Lambda(\theta)$ is
continuously differentiable. }
$$
I(\theta_0):= \lim_{T\to\infty}\frac{1}{T}\int_0^T \big(h^*\dot
\pi^{\theta_0}_t\big)^2dt
$$
is well defined (as the unique limit in $P_{\theta_0}$-probability),
is positive and, moreover,
$$
\lim_{t\to\infty} \E_{\theta_0}\big(h^*\dot
\pi^{\theta_0}_t\big)^2=I(\theta_0),
$$
uniformly on compacts $\K\subset \varTheta$.
\end{enumerate}
Then the MLE  $\hat\theta_T$ is uniformly consistent:
$$
\lim_{T\to\infty}\sup_{\theta_0\in\K}\P_{\theta_0}\Big(\big|\hat
\theta_T-\theta_0\big|\ge \eps\Big)=0, \quad \forall\eps>0,
$$
uniformly asymptotically normal, i.e.
$$
\lim_{T\to\infty}\sup_{\theta_0\in\K}\Big|\E_{\theta_0}f\Big(\sqrt{T}\big(\hat\theta_T-\theta_0\big)\Big)-
\E f(\xi)\Big|=0,\quad \forall f\in \mathcal{C}_b
$$
with a zero mean Gaussian random variable $\xi$, with variance  $1/I(\theta_0)$. Moreover, the
moments converge:
$$
\lim_{T\to\infty}\E_{\theta_0}\big|\sqrt{T}\big(\hat\theta_T-\theta_0\big)\big|^p=\E |\xi|^p,
\quad \forall p>0,
$$
uniformly over compacts $\K\subset\varTheta$.
\end{thm}
Several remarks are in order
\begin{rem}
The condition \eqref{a-1} implies that the chain $S$ is ergodic, but
it is an excessively strong requirement as far as just the
ergodicity is concerned: in fact, $S$ is ergodic if and only  if all
its states communicate (or equivalently the entries of the matrix
exponent $\exp\big(\Lambda\big)$ are all positive). \eqref{a-1}
plays a decisive role in the proof, as it implies appropriate
ergodic properties of the filtering process
$\pi^{\theta}=(\pi^{\theta}_t)_{t\ge 0}$.

The assumptions \eqref{a-2} and \eqref{a-3} are of identifiability
and regularity type and should be checked on the case-to-case basis.
In Section \ref{sec-ex} this is demonstrated with  an example, where
both are verified explicitly in terms of the model data.
\end{rem}

\begin{rem}
The calculation of $\hat \theta_T$ can be quite an
involved numerical optimization problem, which we do not discuss
here. Let us just mention that an effective iterative  EM procedure
for finding a local extremum of $L_T(\theta;X^T)$ was
suggested in \cite{DZ},\cite{ZD} (see also the monograph \cite{EAM} for
additional details). However,  its convergence to the actual value of
$\hat\theta_T$, i.e. to the global minimum, remains vague.
\end{rem}

\subsection{Continuous versus discrete time}

The interest in parameter estimation problems with partial
observations can be traced back at least to the works of L.E.Baum
and T.A.Petrie \cite{BP66}, \cite{P69}, who verified consistency of
MLE for discrete time models with both signal and observation
processes taking finite number of values. The question is very
natural in the context of many engineering problems (see e.g.
\cite{EAM}, \cite{CMR}, \cite{EfMer}). The next major advance has been made
by B.Leroux in \cite{L92}, where the {\em observation} process with general
state space was assumed and consistency of MLE was verified under quite general assumptions.
A partial extension to the signals with general state spaces was recently
reported in \cite{G-CL06}. Spelled in our notations, the main idea is to consider the limit
\begin{equation}
\label{K-L}
\lim_{T\to\infty}\frac 1 T \log L_T(\theta, \theta_0;X^T) =
\lim_{T\to\infty}\frac 1 T \log \frac{d\P_\theta}{d\P_{\theta_0}}_{\big|\F^X_T} = - H(\theta,\theta_0),
\end{equation}
where $H(\theta,\theta_0)$ is the {\em Kullback-Leibler relative entropy} rate between the restrictions of
$\P_\theta$ and $\P_{\theta_0}$ on $\F^X_T$.  If the system is identifiable, $H(\theta,\theta_0)$ attains its
unique minimum at $\theta=\theta_0$ and consistency follows. It is the convergence in \eqref{K-L} and the
verification of the identifiability conditions, which turn to be quite challenging matters.

The  asymptotic normality was established in \cite{BRR98} and the extension to
{\em signals} in general spaces followed in \cite{JP99}. Roughly, the idea is to
expand the likelihood function into powers of the estimation error $\hat \theta_T-\theta_0$,
which vanishes as $T\to\infty$ by consistency, and the proof then amounts to verifying the
appropriate convergence of various residual terms. In continuous time
the direct implementation of this procedure is quite nontrivial as it requires
substitution of the {\em anticipating} random variable $\hat \theta_T$ into
the first argument of $L_T(\theta,\theta_0;X^T)$, which involves the It\^o integral.
Though, in principle, such treatment is possible within the framework of Malliavin calculus,
it would be, perhaps, excessively technical.

We shall prove Theorem \ref{thm} by realizing the program developed
by I.Ibra\-gimov and R.Khasminskii in early 70's \cite{IH}.
The main idea of this approach is to deduce the asymptotic properties
of MLE from the weak convergence of the appropriately scaled likelihoods,
viewed as elements in a function space (more details are given for the reader's convenience
in Section \ref{sec-ih} below). When applied to the large sample
asymptotic problems, this method typically requires good ergodic
properties of the related processes (see e.g. the monograph
\cite{K04}) - in our case, the filtering process
$\pi^\theta=(\pi^\theta_t)_{t\ge 0}$. While for the Kalman-Bucy
linear Gaussian models, such ergodic properties are long known and
are implied by stability of the associated Riccati equation, the
nonlinear case has been studied only during the last decade (see
e.g. an already not quite up to date list of references in
\cite{BCL04}). The role of the ergodic properties of the filtering process
in MLE analysis context (in discrete time) has been first
recognized in \cite{LM97} (see also \cite{LM00a}) and developed further in \cite{DM01}, \cite{DMR04}
(see also \cite{Papa} for a different approach).

The inference of stochastic processes in continuous time is natural in e.g. mathematical
Finance, where the asset prices are thought as positive diffusion processes, such as
geometrical Brownian motion, etc. Though in practice the inference is made from observations,
obtained by sampling the prices at discrete times, the analysis of estimates,
based on the continuous time observations is of practical interest, as it may hint to the
fundamental performance  limitations of the model.

The large sample asymptotic properties of MLE for continuous time models with partial
observations seem to have never been  addressed, beyond the linear Gaussian Kalman-Bucy setting
(see \cite{K84} or e.g. Section 3.1 \cite{K04} for prototypical examples
and the references therein).

Besides of being conceptually appealing in its universality, the
Ibragimov-Khasminskii approach allows to derive stronger properties
of MLE, namely the convergence of moments. To the author's understanding,
the latter was not yet addressed even for the discrete time HMMs.

%Convergence issues of certain estimators for continuous time HMMs
%were studied in \cite{EM}.
%
% The paper \cite{EM} e.g. suggests that the transition probabilities are
% estimated as the ratio between the conditional expectations of the number of transitions
% and the occupation measures, given the observations. However, these filtering probabilities
% require knowledge of the transition probabilities and hence cannot be used for the
% parameter estimation problem at hand!!

As was mentioned before, the computational aspects
of MLE have attracted more attention: e.g. the EM algorithm was implemented in
\cite{DZ,ZD} and \cite{EAM} for the setting, considered in the present
paper.  Some results on recursive parameter estimation
for partially observed diffusions appeared in \cite{LSZ} and \cite{FL}.

Below, in Section \ref{sec-ih}, we proceed with a brief reminder of the Ibragimov-Khasminskii approach.
Section \ref{sec-proof} contains the proof of Theorem \ref{thm} and Section
\ref{sec-ex} presents an example for which the conditions of Theorem \ref{thm} are
verified explicitly. Finally, a concise discussion of the results is
given in Section \ref{sec-dis}.

\subsection*{Notations and conventions} Throughout $C_i$ or $C_{i,j}$, $i,j\in\{1,2,...\}$
denote generic constants, whose precise value is not
important and, moreover, may be different depending on the context
(e.g. in separate claims, proofs, etc.). We shall write $\{x\}_i$ for the
$i$-th entry of the vector $x$. All the statements, involving random
objects, are understood to hold in the  $\P_{\theta_0}$-a.s. sense, if not
mentioned otherwise.

\subsection*{Acknowledgements} I would like to express my gratitude to Yury
Kutoyants, who suggested the problem
and whose advice was crucial for the progress towards its solution.
I would also like to thank  Marina Kleptsyna for many enlightening
discussions and her interest in this work. I am  indebted to them
for their hospitality, without which my stay in France would never
have been the same. Correspondence with Ramon Van Handel was
essential at various  stages of this project and is greatly
appreciated.

\section{The Ibragimov-Khasminskii program} \label{sec-ih}

The main idea of I.Ibragimov and R.Khasminskii, \cite{IH}, is to
consider the sequence of scaled likelihoods
$$
Z_T(u):=L_T(\theta_0+u\varphi_T,\theta_0; X^T), \quad u\in
\mathbb{U}_T:=(\varTheta-\theta_0)/\varphi_T,
$$
where $\varphi_T$ is an appropriate scaling function (in our case
$\varphi_T=1/\sqrt{T}$), as elements from the space $\mathbf{C}_0$
of continuous $\Real\mapsto \Real$ functions, vanishing at $\pm
\infty$, with the norm $\|\psi\|=\sup_{y\in\Real}|\psi(y)|$.
As $Z_T(u)$ is defined only on $\mathbb{U}_T$, its definition is extended
to $\Real$ to make it an element of $\mathbf{C}_0$, in such a way so that its
supremum remains unaltered.

For a measurable set $A\in\Real$
\begin{multline*}
\P_{\theta_0}\big(\sqrt{T}(\hat\theta_T-{\theta_0})\in A\big)=
\P_{\theta_0}\big( \hat\theta_T\in A/\sqrt{T}+{\theta_0}
\big) =\\
 \P_{\theta_0}\left(
\sup_{\eta\in \{A/\sqrt{T}+{\theta_0}\} \cap
\bar\varTheta}L_T(\eta,{\theta_0};X^T) \ge \sup_{\eta\in
\bar\varTheta\setminus \{A/\sqrt{T}+{\theta_0}\}
}L_T(\eta,{\theta_0};X^T)
\right) =\\
\P_{\theta_0}\left( \sup_{u\in A}Z_T(u) \ge
\sup_{u\not \in A}Z_T(u) \right).
\end{multline*}
Suppose that the sequence of random processes $u\mapsto Z_T(u)$,
$T\ge 0$ converges weakly in the function space $\mathbf{C}_0$ to
a random process $Z(u)$ and assume $Z(u)$ attains its maximum at a
unique  point $\hat u$, which has  a continuous distribution (e.g.
Gaussian). Then, as supremum is a continuous functional on
$\mathbf{C}_0$, we have
$$
\P_{\theta_0}\big(\sqrt{T}(\hat\theta_T-{\theta_0})\in A\big)
\xrightarrow{T\to\infty} \P_{\theta_0}\left( \sup_{u\in A}Z(u) \ge
\sup_{u\not \in A}Z(u) \right) = \P_{\theta_0}( \hat{u}\in A ),
$$
In other words, the asymptotic distribution of the scaled estimation
error $\sqrt{T}(\hat\theta_T-{\theta_0})$ converges to the law of
$\hat u$ as $T\to\infty$. The following theorem gives the precise
conditions required for realization of this idea:
\begin{thm}[Theorem 10.1 \cite{IH}]\label{IH-thm}
Let the parameter set $\varTheta$ be an open subset of $\Real$,
functions $u\mapsto Z_T(u)$ be continuous with probability 1
possessing the following properties:
\begin{enumerate}
\item\label{raz} For any compact $\K \subset \varTheta$, there correspond numbers $a$ and $B$ and a positive function $g(u)$, such that
$\lim_{u\to\infty}|u|^Ne^{-g(u)}=0$ for any integer $N$, such that
\begin{enumerate}
\item\label{raz-a} there exist numbers $\alpha>1$ and $m\ge \alpha$ such that for ${\theta_0}\in \K$
$$
\sup_{\stackrel{|u_1|\le R, |u_2|\le R}{u_1,u_2\in \mathbb{U}_T}}
|u_2-u_1|^{-\alpha}
\E_{\theta_0}\big|Z^{1/m}_T(u_2)-Z^{1/m}_T(u_1)\big|^m\le
B(1+R^\alpha)
$$

\item\label{raz-b} For all $u\in\mathbb{U}_T$ and ${\theta_0}\in \K$,
$$
\E_{\theta_0} \sqrt{Z_{T}(u)}\le e^{-g(|u|)}.
$$
\end{enumerate}
\item\label{dva} Uniformly in ${\theta_0}\in \K$ for $T\to\infty$ the marginal (finite-dimensional) distributions of the
random functions $Z_T(u)$ converge to marginal distributions of
random functions  $Z(u)$ where $Z\in \mathbf{C}_0$.
\item\label{tri} The limit functions $Z(u)$ with probability 1 attain the maximum at the unique point $\hat u$
\end{enumerate}
Then uniformly in ${\theta_0}\in \K$ the distribution of random
variables $\sqrt{T}(\hat \theta_T-{\theta_0})$ converge to the
distribution of $\hat u$ and for any continuous loss function $w$
with polynomial growth we have uniformly in ${\theta_0}\in \K$
$$
\lim_{T\to\infty}\E_{\theta_0}w\big(\sqrt{T}(\hat
\theta_T-{\theta_0})\big)=\E w(\hat u).
$$
\end{thm}
The continuity condition \eqref{raz-a} and the large
deviations condition \eqref{raz-b} for the likelihoods tails give
tightness of the probability measures, induced by $Z_T(u)$, while
the convergence of finite dimensional distributions \eqref{dva}
identifies the limit, yielding the aforementioned weak convergence.

\section{The proof}\label{sec-proof}

The proof reduces to verifying the conditions
\eqref{raz}-\eqref{tri} of Theorem \ref{thm} and is preceded by several
important preliminaries. The reader unfamiliar with the material, sketched in the
previous section, is advised to look first at the section \ref{subsec-3.2} below to see how
various propositions are applied.

\subsection{Preliminary results}
The filtering process $\pi^\theta_t$ satisfies the Shiryaev-Wonham
equation (\cite{Sh63},\cite{W}, see also \cite{LSI}):
\begin{equation}
\label{Weq} d\pi^{\theta_0}_t = \Lambda^* \pi^{\theta_0}_t dt +
\big(\pi^{\theta_0}_t\pi^{\theta_0
*}_t-\diag(\pi^{\theta_0}_t)\big)h\big(dX_t -
h^*\pi^{\theta_0}_tdt\big), \quad \pi^{\theta_0}_0=\nu,
\end{equation}
where $\diag(x)$ denotes a scalar matrix with $x\in\Real^d$ on the
diagonal and $h$ stands for a column vector with the entries
$h(a_i)$, $i,...,d$, as before. Having Lipschitz coefficients, this
equation has a strong solution under $\P_{\theta_0}$  as well as
under $\P_{\theta}$, $\theta\ne \theta_0$. Under $\P_{\theta_0}$,
the process $\pi^{\theta_0}_t$ is Markov, since
\begin{equation}\label{innov}
\bar B_t := X_t-\int_0^t h^*\pi^\theta_sds, \quad t\ge 0
\end{equation}
is the {\em innovation} Brownian motion with respect to the
filtration $\F^X_t$.

Further, let $\pi^{\theta_0}_{s,t}(x)$ be the solution of
\eqref{Weq} on the interval $[s,t]$, started at $s$ from
$x\in \osimplex$ ($\osimplex$ is the interior of
$\csimplex$). The map $x\mapsto \pi^{\theta_0}_t(x)$ defines a
stochastic semiflow of smooth diffeomorphisms  (see  Lemma 2.4 in
\cite{ChVH}), which means that on a set of full probability, it is a
smooth injective function of $x$, satisfying the semigroup property
$\pi^{\theta_0}_{0,t}=\pi^{\theta_0}_{s,t}\circ
\pi^{\theta_0}_{0,s}$. We shall also keep the shorter notation $
\pi^{\theta_0}_t=\pi^{\theta_0}_{0,t}(\nu)$, whenever appropriate.
The following facts are central to all the arguments below.

\begin{prop}[Proposition 3.5 in \cite{ChVH}]\label{prop1}
Assume $\lambda_{ij}(\theta)>0$, $i\ne j$, then for
$\F^X_s$-measurable random variables \footnote{throughout $\|\cdot
\|$ stands for the $\ell_1$ norm unless stated otherwise.} $\mu_1,
\mu_2\in \osimplex$
\begin{multline}
\label{pr} \|\pi^{\theta}_{s,t}(\mu_1)-\pi^{\theta}_{s,t}(\mu_2)\|
\le \\ \max_{i=1,...,d}\big(1/\{\mu_1\}_i,1/\{\mu_2\}_i\big)\|\mu_1-\mu_2\| e^{-\gamma(\theta)(t-s)}, \quad
\P_{\theta_0}-a.s.
\end{multline}
where
\begin{equation}
\label{gamma}
\gamma(\theta):= 2\min_{p\ne
q}\sqrt{\lambda_{pq}(\theta)\lambda_{qp}(\theta)}.
\end{equation}
\end{prop}

\begin{rem}
Sometimes it will be more convenient to use the bound
(Corollary 2.3.2 pp. 59 in \cite{Vh})
\begin{equation}
\label{ramon}
\|\pi^{\theta}_{s,t}(\mu_1)-\pi^{\theta}_{s,t}(\mu_2)\| \le
2 e^{-\gamma(\theta)(t-s)}, \quad \P_{\theta_0}-a.s.
\end{equation}
\end{rem}
Let $D\pi^\theta_{s,t}(\mu)\cdot v$ be the directional derivative of
$\pi^\theta_{s,t}(\cdot)$ at a point $\mu\in\osimplex$ in the
direction $v\in\mathcal{T}\osimplex$ (the tangent space to
$\osimplex$).

\begin{prop}[Proposition 3.3 in \cite{ChVH}]
For any $\mu\in\osimplex$ and  $v\in\mathcal{T}\osimplex$
\begin{equation}
\label{Jac} \big\{D\pi^\theta_{s,t}(\mu)\cdot v \big\}_i =
\big\{\pi^\theta_{s,t}(\mu)\big\}_i \sum_{j,k}\frac{v_j}{\mu_k}
\big\{\pi^\theta_{s,t}(\mu)\big\}_k\varphi_{s,t}(i,j,k),\quad
\P_{\theta_0}-a.s.
\end{equation}
where $\varphi_{s,t}(i,j,k)$ is a random process with the property
$$
\max_{i,j,k}\big|\varphi_{s,t}(i,j,k)\big|\le
e^{-\gamma(\theta)(t-s)}.
$$
\end{prop}

Finally we have the following formula,
\begin{prop}[Proposition 2.6 in \cite{ChVH}]
For any $\mu\in \osimplex$,
\begin{equation}
\label{Jacfla}
\pi^{\theta}_{0,t}(\mu)-\pi^{\theta_0}_{0,t}(\mu)=\int_0^t
D\pi^{\theta}_{s,t}\big(\pi^{\theta_0}_s\big)\cdot
\big(\Lambda^*(\theta)-\Lambda^*(\theta_0)\big)
\pi^{\theta_0}_{s}ds.
\end{equation}
\end{prop}
\begin{rem}
The statements of all the three propositions remain valid if
$\theta$ and $\theta_0$ are interchanged. Let us also emphasize that
anything stated $\P_{\theta_0}$-a.s., holds $\P_{\theta}$-a.s. as
well and vice versa.
\end{rem}

First we justify the definition of $g(\theta_0,\theta)$ in
\eqref{a-2} of Theorem \ref{thm}:
\begin{lem}\label{lem-inv}
The pair $(\pi^{\theta_0}_t,\pi^\theta_t)$ is a Markov process  under $\P_{\theta_0}$
and it has a unique invariant measure $\M$. For any Lipschitz $f$ with $\int f d\M=0$
\begin{equation}
\label{toinvm}
\big|\E_{\theta_0}f\big(\pi^{\theta_0}_t, \pi^\theta_t\big)\big| \le C e^{-\frac{1}{2}\gamma(\theta_0)\wedge \gamma(\theta) t},
\end{equation}
with a constant $C$ and  $\gamma(\cdot)$, defined in \eqref{gamma}.
\end{lem}
\begin{proof}
The filtering equation \eqref{Weq} has a unique strong solution, subject to $\pi^{\theta_0}_0=\nu'$ for
any $\nu'\in\osimplex$. If $\nu'$ coincides with $\nu$, the actual distribution of $S_0$, then
the corresponding solution $\pi^{\theta_0}_t$ is the conditional distribution of $S_t$, given $\F^X_t$
and thus the innovation process $\bar B_t=X_t-\int_0^t h^* \pi^{\theta_0}_sds$ is a Brownian motion with
respect to $\F^X_t$. Consequently, $\pi^{\theta_0}_t$ is a Markov process and,
since $\pi^{\theta}_t$ satisfies
\begin{equation}\label{pitheta}
d\pi^{\theta}_t = \Lambda^* \pi^{\theta}_t dt +
\big(\pi^{\theta}_t\pi^{\theta *}_t-\diag(\pi^{\theta}_t)\big)h
\big(d\bar B_t +\big(h^*\pi^{\theta_0}_t-
h^*\pi^\theta_t\big)dt\big),
\end{equation}
the pair $\big(\pi^{\theta_0}_t, \pi^{\theta}_t\big)$ is Markov as well.
Since both processes solve SDEs with Lipschitz coefficients, this pair is also a Feller process
(see e.g. Theorem 3.2 in Ch. III \cite{Khas}), and since it evolves on a compact state space,
at least one invariant measure exists (e.g. Theorem 2.1 Ch. III, \cite{Khas}).

We shall argue for the uniqueness, by showing that if two measures
$\hat \M$ and $\check \M$ are invariant, then  $\int \phi d\hat \M = \int \phi d\check \M$ for any bounded and
continuous $\phi$ and thus $\hat \M$ and $\check \M$ coincide.
For these purposes, we shall explicitly construct the corresponding stationary processes and flows.

Let $(\hat p,\hat q)$ be a random variable with values in $\csimplex \times \csimplex$ and distribution $\hat \M$.
Introduce an $\s$ valued random variable $\hat S_0$ with conditional distribution
$P(\hat S_0=a_i|\hat p,\hat q)=\hat p_i$, $i=1,...,d$. Further, let $\hat S$ be a Markov chain
with transition rates matrix $\Lambda(\theta_0)$ and random initial state $\hat S_0$ and define
the corresponding observation process $\hat X:=\int_0^t h(\hat S_r)dr+B_t$.
Finally, let $(\hat \pi^{\theta_0}_t, \hat \pi^\theta_t)$ be the solutions of \eqref{Weq} and
\eqref{pitheta}, started from $\hat p$ and $\hat q$, respectively, where $X_t$ is replaced with $\hat X_t$.
Then $\hat \pi^{\theta_0}_t$ is nothing but the vector of conditional probabilities
$\P_{\theta_0}\big(\hat S_t=a_i|\F^{\hat X}_t\vee \sigma\{\hat p\}\big)$,
$i=1,...,d$ and thus the corresponding innovation $\hat X_t -\int_0^t h^* \hat \pi^{\theta_0}_rdr$
is a Brownian motion with respect to the filtration
$\F^{\hat X}_t\vee \sigma\{\hat p\}$. Hence  $(\hat \pi^{\theta_0}_t, \hat \pi^\theta_t)$
is a Markov process and it is stationary by construction.

The stationary process $(\check \pi^{\theta_0}_t, \check \pi^\theta_t)$, corresponding to $\check M$,
is defined similarly, but using a Markov chain $\check S$, {\em coupled} with $\hat S$. Namely,
following e.g. \cite{G}, one can construct a Markov chain $(\check S_t,\hat S_t)$ on $\s\times \s$, such that
both $\check S_t$ and $\hat S_t$ are  Markov chains on their own, with the transition rates matrix $\Lambda(\theta_0)$
and initial distributions $\check \mu$ and $\hat \mu$ respectively and, moreover, $\check S_t\equiv \hat S_t$
for any $t\ge \tau$, where  $
\tau =\inf\big\{t: \check S_t = \hat S_t\big\},
$ is the {\em coupling time}, satisfying
\begin{equation}
\label{couple}
\P_{\theta_0}(\tau \ge t) = e^{-\min_{i\ne j}\big(\lambda_{ij}(\theta_0)+\lambda_{ji}(\theta_0)\big) t}\le
e^{-\gamma(\theta_0) t}.
\end{equation}
The observation process $\check X:=\int_0^t h(\check S_r)dr+B_t$ is defined, using the same (!) Brownian motion $B$,
as in the definition of $\hat X_t$. Finally $(\check \pi^{\theta_0}_t, \check \pi^\theta_t)$ denote the solutions
\eqref{Weq} and \eqref{pitheta}, driven by $\check X$ and started from $\check p$ and $\check q$, respectively.

The main point of this arrangement is that after the coupling time the increments of the observation processes
$\hat X_t$ and $\check X_t$ coincide and hence on the set $\{\tau \le s\}$
$$
\hat \pi^{\theta_0}_{s,t}(\cdot)\equiv \check \pi^{\theta_0}_{s,t}(\cdot) \quad \text{and} \quad
\hat  \pi^{\theta}_{s,t}(\cdot)\equiv \check \pi^{\theta}_{s,t}(\cdot), \quad \forall t\ge s
$$
with probability one. Then for any $t\ge s\ge 0$ (w.l.o.g. $|\phi|\le 1$ is assumed)
\begin{multline*}
\left|\int \phi d\hat \M -\int \phi d\check \M \right|= \big|
\E_{\theta_0}\phi\big(\hat \pi^{\theta_0}_t, \hat \pi^{\theta}_t\big)
-\E_{\theta_0}\phi\big(\check \pi^{\theta_0}_t, \check \pi^{\theta}_t\big)
\big|\le \\
\shoveleft
{
\E_{\theta_0}\big|
\phi\big(\hat \pi^{\theta_0}_t, \hat \pi^{\theta}_t\big)
-\phi\big(\check \pi^{\theta_0}_t, \check \pi^{\theta}_t\big)
\big| \one{\tau\ge s}
}
\\
\shoveright
{
+\E_{\theta_0}\big|
\phi\big(\hat \pi^{\theta_0}_t, \hat \pi^{\theta}_t\big)
-\phi\big(\check \pi^{\theta_0}_t, \check \pi^{\theta}_t\big)
\big| \one{\tau< s}\le
}
\\
2\P_{\theta_0}(\tau\ge s)+
\E_{\theta_0}\big|
\phi\big(\hat \pi^{\theta_0}_{s,t}(\hat \pi^{\theta_0}_{s}), \hat \pi^{\theta}_{s,t}(\hat \pi^{\theta}_{s})\big)
-
\phi\big(\check \pi^{\theta_0}_{s,t}(\check \pi^{\theta_0}_{s}), \check \pi^{\theta}_{s,t}(\check \pi^{\theta}_{s})\big)
\big| \one{\tau< s} \le
\\
2\P_{\theta_0}(\tau\ge s)+
\E_{\theta_0}\big|
\phi\big(\hat \pi^{\theta_0}_{s,t}(\hat \pi^{\theta_0}_{s}), \hat \pi^{\theta}_{s,t}(\hat \pi^{\theta}_{s})\big)
-
\phi\big(\hat \pi^{\theta_0}_{s,t}(\check \pi^{\theta_0}_{s}), \hat \pi^{\theta}_{s,t}(\check \pi^{\theta}_{s})\big)
\big|.
\end{multline*}
The latter  can be made arbitrarily small, by taking $s$ and then $t$ large enough and using \eqref{couple} and
\eqref{pr} along with continuity of $\phi$. This verifies uniqueness of the invariant measure of $(\pi^{\theta_0},\pi^\theta)$.

The bound \eqref{toinvm} is derived similarly. We couple the chain $S$ (with initial distribution $\nu$) to the
stationary $\hat S$, which ensures that on the set $\{\tau \le s\}$, the corresponding flows $\pi^{\theta_0}_{s,t}(\cdot)$ and
$\hat \pi^{\theta_0}_{s,t}(\cdot)$ coincide. But then for any $t\ge 0$ (below $L_f$ denotes the Lipschitz constant of $f$)
\begin{multline*}
\big|\E_{\theta_0}f\big(\pi^{\theta_0}_t, \pi^\theta_t\big)\big| =
\big|\E_{\theta_0}f\big(\pi^{\theta_0}_t, \pi^\theta_t\big)-\E_{\theta_0}f\big(\hat \pi^{\theta_0}_t, \hat \pi^\theta_t\big)\big| \le\\
\shoveleft{
2\P_{\theta_0}(\tau \ge t/2) +
}
\\
\shoveright
{
\E_{\theta_0}\big|f\big(\pi^{\theta_0}_{t/2,t}(\pi^{\theta_0}_{t/2}), \pi^{\theta}_{t/2,t}(\pi^{\theta}_{t/2})\big)
-
f\big( \pi^{\theta_0}_{t/2,t}(\hat \pi^{\theta_0}_{t/2}), \pi^{\theta}_{t/2,t}(\hat \pi^{\theta}_{t/2})\big)
\big|\le
}
\\
\shoveleft
{
2\P_{\theta_0}(\tau \ge t/2)+
} \\
\shoveright
{
L_f \big\|\pi^{\theta_0}_{t/2,t}(\pi^{\theta_0}_{t/2})-
\pi^{\theta_0}_{t/2,t}(\hat \pi^{\theta_0}_{t/2})\big\| + L_f
\big\|
\pi^{\theta}_{t/2,t}(\pi^{\theta}_{t/2})
-\pi^{\theta}_{t/2,t}(\hat \pi^{\theta}_{t/2})
\big\|\le
}
\\
2e^{-\gamma(\theta_0)t/2} + 2L_f e^{-\gamma(\theta_0)t/2}+2L_f e^{-\gamma(\theta)t/2}
\le C e^{-\frac 1 2 \gamma(\theta_0)\wedge \gamma(\theta) t},
\end{multline*}
with a constant $C>0$.
\end{proof}
The combination of the formulae \eqref{Jac} and \eqref{Jacfla},
involves $1/\pi^{\theta_0}_s$, which is $\P_{\theta_0}$-a.s. bounded
on any finite interval (see Corollary 2.2 in \cite{ChVH}). However,
under assumption \eqref{serg}, $\pi^{\theta_0}_s$ is repelled from
the boundary of $\osimplex$ strongly enough to guarantee the
following uniform integrability:
\begin{lem}
Assume \eqref{serg}, then for any $\mu\in\osimplex$
\begin{equation}
\label{ui} \sup_{t\ge s}\E_{\theta_0}\left(\frac{1}{\min_i
\big\{\pi^{\theta_0}_{s,t}(\mu)\big\}_i}\right)^m<\infty, \quad
m=1,2,...
\end{equation}
uniformly over $\theta_0\in \bar\varTheta$.
\end{lem}
\begin{proof}
The proof follows the arguments of Proposition 3.7 in \cite{ChVH},
which verifies \eqref{ui} for $m=1$. As the equation \eqref{Weq} is
time homogeneous, no generality is lost if we assume $s=0$ (and use
the shorter notation
$\pi^i_t:=\big\{\pi^{\theta_0}_{0,t}(\mu)\big\}_i$). By Lemma 3.6
\cite{ChVH}, for any $m=1,2,...$ and $T>0$
\begin{equation}\label{fin}
\E_{\theta_0}\int_0^T \big(\pi^i_t\big)^{-m}dt<\infty.
\end{equation}
By the It\^o formula
\begin{multline*}
\big(\pi^i_t\big)^{-m} = (\mu^i)^{-m} -\int_0^t m(\pi^i_s)^{-m-1}\sum_{j\ne i}\lambda_{ji}\pi^j_sds +\\
\int_0^t \Big( m|\lambda_{ii}| (\pi^i_s)^{-m}
-m(\pi^i_s)^{-m}\big(h^*\pi_s-h^i\big)\big(h(S_s)-h^*\pi_s\big)
+\\
\frac 1 2 m(m+1)(\pi^i_s)^{-m}\big(h^*\pi_s-h^i\big)^2\Big)ds -
\int_0^t m(\pi^i_s)^{-m}\big(h^*\pi_s-h^i\big)dB_s.
\end{multline*}
Set $M_t: =\E_{\theta_0} (\pi^i_t)^{-m}$, then by the Jensen
inequality
$$
\E_{\theta_0} (\pi^i_t)^{-m-1}\ge M_t^{1+1/m}.
$$
By \eqref{fin}, the expectation of the stochastic integral vanishes
and, since $\min_{j\ne i}\lambda_{ji}>0$ is assumed, we have
$$
\frac{d}{dt}M_t \le    -  K_1 M_t^{1+1/m} + K_2 M_t
$$
with  constants $K_1>0$ and $K_2$. For any fixed $m$, this
differential inequality implies $\sup_{t\ge 0}M_t<\infty$, which is
nothing but \eqref{ui}.
\end{proof}
\begin{rem}
Clearly, the statement of the lemma remains valid for $\pi^\theta$,
$\theta\ne \theta_0$ i.e.
$$
\sup_{t\ge s}\E_{\theta_0}\left(\frac{1}{\min_i
\big\{\pi^{\theta}_{s,t}(\mu)\big\}_i}\right)^m<\infty, \quad
m=1,2,...
$$
uniformly over $\theta,\theta_0\in \bar\varTheta$.
\end{rem}
The following lemma is an extension (in the case of unperturbed $h$)
of Theorem 1.1. from \cite{ChVH}:
\begin{lem}\label{roblem}
Assume \eqref{serg}, then  for any $\mu\in \osimplex$
and uniformly over $\theta_0,\theta\in\bar\varTheta$,
\begin{equation}
\label{robust} \sup_{t\ge
s}\E_{\theta_0}\big\|\pi^{\theta_0}_{s,t}(\mu)-\pi^{\theta}_{s,t}(\mu)\big\|^m
\le C |\theta_0-\theta|^m, \quad m=1,2,...
\end{equation}
with a constant $C>0$, possibly dependent on $m$.
\end{lem}
\begin{proof}
Using \eqref{Jac} and \eqref{Jacfla},
\begin{align*}
&\E_{\theta_0}\big\|\pi^{\theta_0}_{s,t}(\mu)-\pi^{\theta}_{s,t}(\mu)\big\|^m  \le \\
&C_1 \big\|\Lambda(\theta_0)-\Lambda(\theta)\big\|^m
\E_{\theta_0}\left(\int_s^t
 \frac{e^{-\gamma(\theta)(t-r)}dr}{\min_k \big\{\pi^{\theta_0}_{s,r}(\mu)\big\}_k}
\right)^m \stackrel{\dagger}{\le} \\
& C_1 \big\|\Lambda(\theta_0)-\Lambda(\theta)\big\|^m
\gamma^{1-m}(\theta) \int_s^t
 e^{-\gamma(\theta)(t-r)}\E_{\theta_0}\left(\frac{1}{\min_k \big\{\pi^{\theta_0}_{s,r}(\mu)\big\}_k}
\right)^m dr\le \\
& \frac{C_2 }{\gamma^{m}(\theta)}
\big\|\Lambda(\theta_0)-\Lambda(\theta)\big\|^m\le C
|\theta_0-\theta|^m,
\end{align*}
where $\dagger$ is the  Jensen inequality and  the last bound is valid as $\Lambda(\theta)$ is continuously
differentiable on $\bar\varTheta$.
\end{proof}
Finally we shall need the following law of large numbers:
\begin{lem}
Under the assumption \eqref{a-1},
\begin{multline}
\label{key}
\E_{\theta_0}\left(\frac 1 T\int_0^T\big(h^*\pi^{\theta_0}_t-h^*\pi^\theta_t\big)^2dt - g(\theta_0,\theta)\right)^{2k}\le\\
\frac{C_1 |\theta_0-\theta|^{4k}}{T^k}+ \frac{C_2}{T^{2k}},\quad k=1,2,...
\end{multline}
with constants $C_1$ and $C_2$, possibly dependent on $k$.
\end{lem}

\begin{proof}
Let
$g_t(\theta_0,\theta):=\E_{\theta_0}\big(h^*\pi^{\theta_0}_t-h^*\pi^{\theta}_t\big)^2$,
then
\begin{multline}\label{deux}
\E_{\theta_0}\left(\frac 1 T\int_0^T\big(h^*\pi^{\theta_0}_t-h^*\pi^\theta_t\big)^2dt - g(\theta_0,\theta)\right)^{2k}\le \\
2^{2k-1}\E_{\theta_0}\left(\frac 1
T\int_0^T\Big(\big(h^*\pi^{\theta_0}_t-h^*\pi^\theta_t\big)^2 -
g_t(\theta_0,\theta)\Big)dt\right)^{2k}
+\\
2^{2k-1}\left(\frac 1 T\int_0^T
\Big(g_t(\theta_0,\theta)-g(\theta_0,\theta)\Big)dt\right)^{2k}.
\end{multline}
The second term in the right hand side of \eqref{deux} contributes
$C_2/T^{2k}$ in \eqref{key}, since by \eqref{toinvm},
$g_t(\theta_0,\theta)$ converges to
$g(\theta_0,\theta)$ exponentially fast. The contribution of the first term  in \eqref{deux}
is deduced from a version of Lemma 2.1 in \cite{Kh66}. In
particular, this lemma implies that if a zero mean process $\Phi_t$,
has a bounded  moment of order $2k+\delta$ for some $\delta>0$ and
is a strong mixing with the  coefficient $\alpha(\tau)$, decaying to
zero sufficiently fast as $\tau\to\infty$, then
$$
\E \left(\int_0^T \Phi_t dt\right)^{2k} \le C T^k,
$$
with a constant $C>0$, depending on the moments of $\Phi_t$. This is
precisely the type of estimate needed for \eqref{key}, however, it
is not clear whether  $(\pi^{\theta_0}_t, \pi^\theta_t)$ is a strong
mixing. Note that \eqref{pr} (with $\theta$, replaced by $\theta_0$)
does not necessarily imply that $\pi^{\theta_0}_t$ is a strong mixing,
as it does not even guarantee that the distribution of $\pi^{\theta_0}_t$
converges to the invariant measure in total variation norm (only weak convergence follows).
Fortunately, the strong mixing property is not crucial for the claim
of this lemma and it can be modified to suit our purposes. The exact
formulation of an analogous statement, namely Lemma \ref{khas-lem},
and its proof are given in  Appendix \ref{sec-app}.

We aim to apply  Lemma \ref{khas-lem} to the process
$\Phi(t):=\big(h^*\pi^{\theta}_t-h^*\pi^{\theta_0}_t\big)^2-g_t(\theta_0,\theta)$.
By the definition $\E_{\theta_0}\Phi(t)\equiv 0$ and by Lemma \ref{roblem}, the
condition \eqref{mom} is satisfied with $b:=(\theta_0-\theta)^2$. So
to prove
\begin{equation}
\label{une} \E_{\theta_0}\left(\frac 1
T\int_0^T\Big(\big(h^*\pi^{\theta_0}_t-h^*\pi^\theta_t\big)^2 -
g_t(\theta_0,\theta)\Big)dt\right)^{2k} \le \frac{C
|\theta_0-\theta|^{4k}}{T^k},
\end{equation}
we shall show that  \eqref{kakmix} holds, i.e. for any $n\ge 2$
$$
\left|\E_{\theta_0}\Phi(t_1)...\Phi(t_n) - \E\Phi(t_1)...\Phi(t_i)\E
\Phi(t_{i+1})...\Phi(t_n)\right|\le C_n b^{n}\alpha(t_{i+1}-t_i)
$$
with  exponential $\alpha(\tau)$.

By the formula \eqref{Jacfla} (with $\theta$ and $\theta_0$
interchanged), for $s\le t$
\begin{multline*}
\pi^{\theta_0}_t-\pi^{\theta}_t=\int_0^s
D\pi^{\theta_0}_{r,t}\big(\pi^{\theta}_r\big)\cdot
\big(\Lambda^*(\theta_0)-\Lambda^*(\theta)\big)
\pi^{\theta}_{r}dr + \\
\int_s^t
D\pi^{\theta_0}_{r,t}\big(\pi^{\theta}_{s,r}(\pi^{\theta}_s)\big)\cdot
\big(\Lambda^*(\theta_0)-\Lambda^*(\theta)\big)
\pi^{\theta}_{s,r}(\pi^{\theta}_s)dr :=
I_s^t+J_{s,t}(\pi^{\theta}_s)
\end{multline*}
Recall that the pair $(\pi^{\theta_0}, \pi^{\theta})$ is a Markov
process under $\P_{\theta_0}$ and let $\F^\pi_t$ denote its natural
filtration. Using  \eqref{Jac} and \eqref{ui}, we get
\begin{multline}\label{Ipart}
\Big(\E_{\theta_0}\big\|I_s^t\big\|^m\Big)^{1/m} \le \\
\bigg[C_1\big\|\Lambda(\theta_0)-\Lambda(\theta)\big\|^m
\E_{\theta_0}\bigg(\int_0^s
\frac{1}{\min_k\{\pi^{\theta}_r\}_k}e^{-\gamma(\theta_0)(t-r)}dr\bigg)^m
\bigg]^{1/m}\\
\le C_2|\theta_0-\theta| e^{-\gamma(\theta_0)(t-s)},\quad m=1,2,...
\end{multline}
where the latter inequality is deduced as in the proof of Lemma
\ref{roblem}. Similarly we have
\begin{equation}
\label{Jpart-ub}
\Big[\E_{\theta_0}\big\|J_{s,t}(\pi^{\theta}_s)\big\|^m\Big]^{1/m}
\le C_3 |\theta_0-\theta|.
\end{equation}

Further, for any $\mu_1,\mu_2\in\osimplex$,
\begin{align*}
&J_{s,t}(\mu_1)-J_{s,t}(\mu_2)=\int_s^t
D\pi^{\theta_0}_{r,t}\big(\pi^{\theta}_{s,r}(\mu_1)\big)\cdot
\big(\Lambda^*(\theta_0)-\Lambda^*(\theta)\big)
\pi^{\theta}_{s,r}(\mu_1)dr- \\
& \int_s^t
D\pi^{\theta_0}_{r,t}\big(\pi^{\theta}_{s,r}(\mu_2)\big)\cdot
\big(\Lambda^*(\theta_0)-\Lambda^*(\theta)\big)
\pi^{\theta}_{s,r}(\mu_2)dr =\\
& \int_s^t
D\pi^{\theta_0}_{r,t}\big(\pi^{\theta}_{s,r}(\mu_1)\big)\cdot
\big(\Lambda^*(\theta_0)-\Lambda^*(\theta)\big)
\big(\pi^{\theta}_{s,r}(\mu_1)-\pi^{\theta}_{s,r}(\mu_2)\big)dr +\\
& \int_s^t
\Big(D\pi^{\theta_0}_{r,t}\big(\pi^{\theta}_{s,r}(\mu_1)\big)-
D\pi^{\theta_0}_{r,t}\big(\pi^{\theta}_{s,r}(\mu_2)\big) \Big) \cdot
\big(\Lambda^*(\theta_0)-\Lambda^*(\theta)\big)
\pi^{\theta}_{s,r}(\mu_2)dr\\
& := R+Q
\end{align*}
The bound \eqref{ramon} and the formula \eqref{Jac} yields for any
$m=1,2,...$
\begin{multline*}
\Big(\E_{\theta_0}\|R \|^m\Big)^{1/m} \le
C_4\big\|\Lambda(\theta_0)-\Lambda(\theta)\big\| \times \\
\bigg[\E_{\theta_0}\bigg(\int_s^t \frac{1}{\min_k
\{\pi^{\theta}_{s,r}(\mu_1)\}_k}
e^{-\gamma(\theta_0)(t-r)}e^{-\gamma(\theta)(s-r)}dr\bigg)^m\bigg]^{1/m}\le\\
C_5|\theta_0-\theta| \exp\Big\{-\frac 1 2[\gamma(\theta_0)\wedge
\gamma(\theta)](t-s)\Big\}.
\end{multline*}
Using \eqref{Jac} (and  utilizing its particular dependence on
$\mu$) and \eqref{ramon}
\begin{multline*}
\Big(\E_{\theta_0}\|Q\|^m\Big)^{1/m} \le  C_6 \big\|\Lambda(\theta_0)-\Lambda(\theta)\big\| \times \\
\bigg[\E_{\theta_0}\bigg(\int_s^t \frac{1}{\min_k
\{\pi^{\theta}_{s,r}(\mu_2)\}_k }\frac{1}{\min_k
\{\pi^{\theta}_{s,r}(\mu_1)\}_k }
e^{-\gamma(\theta_0)(t-r)}e^{-\gamma(\theta)(r-s)} dr
\bigg)^m\bigg]^{1/m} \\
\le C_7|\theta_0-\theta| \exp\Big\{-\frac 1 2[\gamma(\theta_0)\wedge
\gamma(\theta)](t-s)\Big\}.
\end{multline*}
Hence, for any $\mu_1,\mu_2\in\osimplex$
\begin{multline}
\label{Jpart}
\Big[\E_{\theta_0}\|J_{s,t}(\mu_1)-J_{s,t}(\mu_2)\|^m\Big]^{1/m} \le
\\
C_8 |\theta_0-\theta| \exp\Big\{-\frac 1 2[\gamma(\theta_0)\wedge
\gamma(\theta)](t-s)\Big\}
\end{multline}

Below we shall use the fact, that if $\xi_1,...,\xi_m$ are random variables (depending on
a parameter $b>0$), such that $\left(\E|\xi_i|^k\right)^{1/k}\le C_{i,k} b$ for any $k\ge 1$, then
by the H\"older inequality for integers $k_1,...,k_m$
\begin{equation}\label{hold}
\E |\xi_1|^{k_1}...|\xi_m|^{k_m} \le C |b|^{k_1+...+k_m},
\end{equation}
with a constant depending on $k_i$'s and $m$.

For any $s\le t$,
\begin{multline*}
\Phi_t=\big(h^*\pi^{\theta}_t-h^*\pi^{\theta_0}_t\big)^2-g_t(\theta_0,\theta)
=
\Big(h^*I_s^t+h^*J_{s,t}(\pi^{\theta}_s)\Big)^2-g_t(\theta_0,\theta)=\\
h^*I_s^t\big[h^*I_s^t+2h^*J_{s,t}(\pi^{\theta}_s)\big]+\Big(h^*J_{s,t}(\pi^{\theta}_s)\Big)^2-g_t(\theta_0,\theta)\\:=
\psi_s^t + \phi_{s,t}(\pi^{\theta}_s)
\end{multline*}
The inequalities \eqref{hold} and  \eqref{Ipart} imply
\begin{equation}\label{psib}
\E_{\theta_0}\psi_s^t:=\E_{\theta_0}h^*I_s^t\big[h^*I_s^t+2h^*J_{s,t}(\pi^{\theta}_s)\big]
\le C_9 |\theta_0-\theta|^2 e^{-\gamma(\theta_0)(t-s)}.
\end{equation}
Further,
\begin{multline}\label{kakmix3}
\left|\E_{\theta_0} \Phi_{t_1}...\Phi_{t_n} - \E_{\theta_0}\Phi_{t_1}...\Phi_{t_i}\E_{\theta_0} \Phi_{t_{i+1}}...\Phi_{t_n}\right|=\\
\left|\E_{\theta_0}\Phi_{t_1}...\Phi_{t_i}\Big(\E_{\theta_0}
\big(\Phi_{t_{i+1}}...\Phi_{t_n}\big|\F^\pi_{t_i}\big) - \E_{\theta_0}
\Phi_{t_{i+1}}...\Phi_{t_n}\Big)\right|.
\end{multline}
Substituting $\Phi_{t_j}= \psi_{t_i}^{t_j} +
\phi_{t_i,t_j}(\pi^\theta_{t_i})$ for $j=i+1,...,n$, expanding all
the expressions into monomials and using the bound \eqref{psib}, we
see that the right hand side of \eqref{kakmix3} is bounded by a sum
of terms of the form $C_{ij}|\theta_0-\theta|^{4n}
e^{-\gamma(\theta_0)(t_j-t_i)}$, $j=i+1,...,n$ (and hence altogether
bounded by $ C|\theta_0-\theta|^{4n}
e^{-\gamma(\theta_0)(t_{i+1}-t_i)}$ for some $C>0$) and the term
\begin{multline}
\label{votkak}
\bigg|\E_{\theta_0}\Phi_{t_1}...\Phi_{t_i}\Big(\E_{\theta_0} \big(\phi_{t_i,t_{i+1}}(\pi^{\theta}_{t_i})...\phi_{t_i,t_{i+1}}(\pi^{\theta}_{t_i})\big|\F^\pi_{t_i}\big) -\\
\shoveright
{
\E_{\theta_0} \phi_{t_i,t_{i+1}}(\pi^{\theta}_{t_i})...\phi_{t_i,t_{i+1}}(\pi^{\theta}_{t_i})\Big)\bigg|=
}
\\
\shoveleft
{
\bigg|\E_{\theta_0}\Phi_{t_1}...\Phi_{t_i}\Big(\phi_{t_i,t_{i+1}}(\pi^{\theta}_{t_i})...\phi_{t_i,t_{i+1}}(\pi^{\theta}_{t_i}) -
}
\\
\shoveright
{
\tilde\E_{\theta_0} \phi_{t_i,t_{i+1}}(\tilde\pi^{\theta}_{t_i})...\phi_{t_i,t_{i+1}}(\tilde\pi^{\theta}_{t_i})\Big)\bigg|\le
}
\\
\shoveleft{
\E_{\theta_0}\big|\Phi_{t_1}...\Phi_{t_i}\big|
\tilde \E_{\theta_0}\Big|\phi_{t_i,t_{i+1}}(\pi^{\theta}_{t_i})...\phi_{t_i,t_{i+1}}(\pi^{\theta}_{t_i}) -
}
\\
\phi_{t_i,t_{i+1}}(\tilde\pi^{\theta}_{t_i})...\phi_{t_i,t_{i+1}}(\tilde\pi^{\theta}_{t_i})\Big|,
\end{multline}
where $\tilde \E_{\theta_0}$ denotes  expectation over an auxiliary
probability space $(\tilde{\Omega}, \tilde{\F}, \tilde{\P})$, on which
$\tilde{\pi}^{\theta}$ is defined as a copy of $\pi^{\theta}$.

Using the elementary summation formula
\begin{align*}
f_1(x)...f_n(x)-f_1(y)...f_n(y)= &\big\{f_1(x)-f_1(y)\big\}f_2(x)...f_n(x)+\\
f_1(y) &\big\{f_2(x)-f_2(y)\big\}f_3(x)...f_n(x)+\\
 &\quad \quad \quad...\\
  f_1(y)...f_{n-1}(y)&\big\{f_n(x)-f_n(y)\big\},
\end{align*}
and the bounds \eqref{Jpart-ub}, \eqref{Jpart} and \eqref{hold}, we conclude that
the expression in \eqref{votkak} is bounded by a sum of terms
of the form
$$
C_{i,j}|\theta_0-\theta|^{4n}\exp\Big\{-\frac 1
2[\gamma(\theta_0)\wedge \gamma(\theta)](t_j-t_i)\Big\}, \quad
j=i+1,...,n
$$
and hence by
$$
C |\theta_0-\theta|^{4n}\exp\Big\{-\frac 1 2[\gamma(\theta_0)\wedge
\gamma(\theta)](t_{i+1}-t_i)\Big\}
$$
with some $C>0$. This verifies the condition \eqref{kakmix} of Lemma
\ref{khas-lem}, which yields  \eqref{une} and in turn the required
bound \eqref{key}.
\end{proof}

\subsection{The proof of Theorem \ref{thm}}\label{subsec-3.2}
The proof verifies the conditions of Theorem \ref{IH-thm} and follows
the lines of the proof of Theorem 2.8 in \cite{K04} with the
adjustments, based on the properties, derived in the preceding
section.
\begin{lem}
Assume \eqref{a-1} of Theorem \ref{thm}, then \eqref{raz-a} of
Theorem \ref{IH-thm} holds for any even $m\ge 2$.
\end{lem}
\begin{proof}
For an integer $k\ge 1$, define $V_T :=
\left(\dfrac{Z_T(u_2)}{Z_T(u_1)}\right)^{1/2k}$, then
\begin{multline*}
\E_{\theta_0} \left( Z_T(u_1)^{1/2k}-Z_T(u_2)^{1/2k}
\right)^{2k} =\\
\E_{\theta_0} Z_T(u_1)\left(1-V_T\right)^{2k} =
\E_{\theta_{u_1}^T}\left(1-V_T\right)^{2k},
\end{multline*}
where the notation $\theta_{u_1}^T=\theta_0+u_1/\sqrt{T}$ is used for brevity.
Recall that
$$
\frac{Z_T(u_2)}{Z_T(u_1)}= \exp\left\{ \int_0^T
h^*\Big(\pi_t^{\theta_{u_2}^T}-\pi_t^{\theta_{u_1}^T}\Big)d\widetilde
B_t- \frac 1 2 \int_0^T\Big( h^*\pi_t^{\theta_{u_2}^T}-
h^*\pi_t^{\theta_{u_1}^T}\Big)^2dt \right\},
$$
where $\widetilde B=(\widetilde B_t)_{t\ge 0}$ is the innovation
Brownian motion  under $\P_{\theta_{u_1}^T}$. Let $\delta_t :=
h^*\pi_t^{\theta_{u_2}^T}- h^*\pi_t^{\theta_{u_1}^T}$, then by the
It\^o formula
$$
V_T = 1+ \frac{1}{2k} \int_0^T V_t \delta_t d\widetilde B_t
-\frac{1}{4k}\Big(1-\frac {1}{2k}\Big) \int_0^T V_t \delta_t^2dt,
$$
and hence
\begin{multline*}
\E_{\theta_{u_1}^T}\big(1-V_T\big)^{2k} \le 2^{2k-1}
\E_{\theta_{u_1}^T}\left(\frac{1}{2k} \int_0^T V_t \delta_t d\widetilde B_t\right)^{2k}+
\\
\shoveright{
2^{2k-1} \E_{\theta_{u_1}^T}\left( \frac{2k-1}{8 k^2}
\int_0^T V_t \delta_t^2dt
\right)^{2k} \le
}
\\
\shoveleft{
C_1 T^{k-1}\int_0^T \E_{\theta_{u_1}^T} (V_t\delta_t\big)^{2k}  dt
+ C_2 T^{2k-1}\int_0^T \E_{\theta_{u_1}^T} \big(V_t
\delta_t^2\big)^{2k}dt
=
} \\
\shoveleft{
C_1 T^{k-1}\int_0^T \E_{\theta_{u_2}^T} \big(\delta_t\big)^{2k}
dt + C_2 T^{2k-1}\int_0^T \E_{\theta_{u_2}^T} \big(
\delta_t\big)^{4k}dt\le
}
\\
C_3  (u_1-u_2)^{2k} + C_4  (u_1-u_2)^{4k},
\end{multline*}
where the bound \eqref{robust} has been used in the latter
inequality. This implies
$$
(u_1-u_2)^{-2k} \E_{\theta_0} \left(
Z_T(u_1)^{1/2k}-Z_T(u_2)^{1/2k} \right)^{2k} \le
C_3(1+R^{2k}),
$$
with a constant $C_3$, depending only on the compact $\K$ and $k$,
as required.
\end{proof}

\begin{lem}
Under the assumptions of Theorem \ref{thm}, \eqref{raz-b} of Theorem
\ref{IH-thm} holds.
\end{lem}
\begin{proof}
Instead of \eqref{raz-b} we shall verify the sufficient condition
\begin{equation}
\label{instead} \P_{\theta_0}\left(Z_T(u)\ge e^{-\kappa
|u|/4}\right)\le \frac{C_1}{u^{2m}}, \quad \text{for any integer\ }
m\ge 1.
\end{equation}
Indeed, by the Cauchy-Schwartz inequality
\begin{multline*}
\E_{\theta_0} \sqrt{Z_T(u)} \le \E_{\theta_0} \one{Z_T(u)\ge
e^{-\kappa|u|/4}}\sqrt{Z_T(u)} + e^{-\kappa |u|/8}
\le \\
 \sqrt{\P_{\theta_0}\left(Z_T(u)\ge e^{-\kappa|u|/4}\right)}
+ e^{-\kappa |u|/8}\le \sqrt{\frac{C_1}{|u|^{2m}}} + e^{-\kappa
|u|/8}\le C_2 e^{-m\log|u|}
\end{multline*}
and since the latter is required to hold for {\em any} $m\ge 1$,
\eqref{raz-b} of Theorem \ref{IH-thm} holds as well.

The formula \eqref{Jacfla} (with $\theta$ and $\theta_0$  swapped)
implies
$$
h^*\pi^{\theta}_t-h^*\pi^{\theta_0}_t = \int_0^t h^*
D\pi^{\theta_0}_{s,t}\big(\pi^{\theta}_s\big)\cdot
\big(\Lambda^*(\theta)-\Lambda^*(\theta_0)\big) \pi^{\theta}_{s}ds,
$$
and as $\Lambda(\theta)$ has a continuous second derivative
$\Lambda''(\theta)$
\begin{multline*}
h^*\pi^{\theta}_t-h^*\pi^{\theta_0}_t = (\theta-\theta_0)\int_0^t
h^* D\pi^{\theta_0}_{s,t}\big(\pi^{\theta}_s\big)\cdot
\Lambda'^*(\theta_0)
\pi^{\theta}_{s}ds + \\
\frac 1 2(\theta-\theta_0)^2 \int_0^t h^*
D\pi^{\theta_0}_{s,t}\big(\pi^{\theta}_s\big)\cdot
\Lambda''^*(\tilde{\theta}) \pi^{\theta}_{s}ds \\=:
(\theta-\theta_0)\alpha_t(\theta_0,\theta)
+(\theta-\theta_0)^2\beta_t(\theta_0,\theta)
\end{multline*}
with $\tilde \theta\in [\theta_0,\theta]$. Due to the property
\eqref{Jac}, $\sup_{t\ge 0}\E_{\theta_0}
\big|\alpha_t(\theta_0,\theta)\big|^2<\infty$ and $\sup_{t\ge
0}\E_{\theta_0} \big|\beta_t(\theta_0,\theta)\big|^2<\infty$, hence
\begin{equation}\label{pochti}
g_t(\theta_0,\theta) = (\theta-\theta_0)^2
\E_{\theta_0}\big(\alpha_t(\theta_0,\theta)\big)^2 +
o\big((\theta_0-\theta)^2\big),
\end{equation}
where $o(\cdot)$ is uniform in $t\ge 0$.
Note that $\alpha_t(\theta_0,\theta_0)=h^*\dot\pi^{\theta_0}_t$ and
\begin{multline*}
\alpha_t(\theta_0,\theta_0)-\alpha_t(\theta_0,\theta) = \int_0^t
h^* D\pi^{\theta_0}_{s,t}\big(\pi^{\theta_0}_s\big)\cdot
\Lambda'^*(\theta_0)
\pi^{\theta_0}_{s}ds-\\
\shoveright
{
\int_0^t h^*
D\pi^{\theta_0}_{s,t}\big(\pi^{\theta}_s\big)\cdot
\Lambda'^*(\theta_0)
\pi^{\theta}_{s}ds=
}\\
\shoveleft{
\int_0^t h^* D\pi^{\theta_0}_{s,t}\big(\pi^{\theta_0}_s\big)\cdot
\Lambda'^*(\theta_0)
\big(\pi^{\theta_0}_{s}-\pi^\theta_s\big)ds+
}
\\
\int_0^t h^*
\Big(D\pi^{\theta_0}_{s,t}\big(\pi^{\theta_0}_s\big)-D\pi^{\theta_0}_{s,t}\big(\pi^{\theta}_s\big)\Big)\cdot
\Lambda'^*(\theta_0) \pi^\theta_sds
\end{multline*}
Now using the formula \eqref{Jac} and the bound \eqref{robust}, we
find that
$$
\sup_{t\ge
0}\E_{\theta_0}\Big(\alpha_t(\theta_0,\theta_0)-\alpha_t(\theta_0,\theta)
\Big)^2 \le C_3(\theta_0-\theta)^2.
$$
This and \eqref{pochti} imply
\begin{multline*}
g_t(\theta_0,\theta) =
(\theta-\theta_0)^2 \E_{\theta_0}\big(\alpha_t(\theta_0,\theta_0)\big)^2 + o\big((\theta_0-\theta)^2\big)=\\
(\theta-\theta_0)^2 \E_{\theta_0}\big(h^*\dot\pi^{\theta_0}_t\big)^2
+o\big((\theta_0-\theta)^2\big)
\end{multline*}
and hence, by the assumption \eqref{a-3} and \eqref{toinvm},
\begin{multline*}
g(\theta_0,\theta) =\lim_{t\to\infty}g_t(\theta_0,\theta) =
(\theta_0-\theta)^2 \lim_{t\to\infty}
\E_{\theta_0}\big(h^*\dot\pi^{\theta_0}_t\big)^2+o\big((\theta_0-\theta)^2\big)
=\\ (\theta_0-\theta)^2 I(\theta_0) +
o\big((\theta_0-\theta)^2\big).
\end{multline*}
Thus for some small enough $r>0$
\begin{equation}\label{gg}
g(\theta_0, \theta)\ge \frac 1 2 I(\theta_0) (\theta_0-\theta)^2
,\quad \forall |\theta_0-\theta|\le r,
\end{equation}
uniformly over $\theta_0\in \K$. Since $ q(r) : = \inf_{\theta_0\in
\K}\inf_{|\theta_0-\theta|\ge r} g(\theta_0,\theta) $ is strictly
positive by the assumption \eqref{a-2}, we have
$$
g(\theta_0,\theta) \ge q(r)
\frac{(\theta_0-\theta)^2}{|\varTheta|^2}, \quad \forall
|\theta_0-\theta|\ge r,
$$
uniformly over $\theta_0\in \K$, where $|\varTheta|$ denotes the
diameter of $\varTheta$. Hence, with $\kappa:= r\wedge q(r)>0$,
$$
g(\theta_0,\theta)> \kappa (\theta_0-\theta)^2, \quad \forall
\theta_0, \theta\in \K.
$$
In particular, we have $ T g(\theta_0,\theta_u^T) \ge \kappa u^2, $
whenever $u$ belongs to a compact in $\mathbb{U}_T$. Further
\begin{multline*}
\P_{\theta_0}\left(Z_T(u)\ge e^{-\kappa u^2/4}\right) = \\
\shoveleft{
\P_{\theta_0}\left( \int_0^T
h^*\Big(\pi_t^{\theta_u^T}-\pi_t^{\theta_0}\Big)d\bar B_t - \frac 1
2\int_0^T \Big(h^*\pi_t^{\theta_u^T}-h^*\pi_t^{\theta_0}\Big)^2
dt\ge -\frac{\kappa}{4} u^2
\right) =
}
\\
\shoveleft
{
\P_{\theta_0}\bigg( \int_0^T
h^*\Big(\pi_t^{\theta_u^T}-\pi_t^{\theta_0}\Big)d\bar B_t -
}
\\
\shoveright
{
 \frac 1
2\int_0^T \Big[\Big(h^*\pi_t^{\theta_u^T}-h^*\pi_t^{\theta_0}\Big)^2
-g(\theta_0,\theta_u^T) \Big]
dt\ge
-\frac{\kappa}{4}u^2 + \frac T 2
g(\theta_0,\theta_u^T)
\bigg)\le
} \\
\shoveleft{
\P_{\theta_0}\bigg(
\int_0^T h^*\Big(\pi_t^{\theta_u^T}-\pi_t^{\theta_0}\Big)d\bar B_t -
}
\\
\shoveright
{
\frac 1 2\int_0^T
\Big[\Big(h^*\pi_t^{\theta_u^T}-h^*\pi_t^{\theta_0}\Big)^2
-g(\theta_0,\theta_u^T) \Big] dt\ge \frac{\kappa}{4}u^2
\bigg) \le
}
\\
\shoveleft{
\P_{\theta_0}\left( \bigg|\int_0^T
h^*\Big(\pi_t^{\theta_u^T}-\pi_t^{\theta_0}\Big)d\bar B_t\bigg|\ge
\frac{\kappa}{8}u^2
\right)+
}
\\
\P_{\theta_0}\left(\bigg|\frac 1 2\int_0^T
\Big[\Big(h^*\pi_t^{\theta_u^T}-h^*\pi_t^{\theta_0}\Big)^2
-g(\theta_0,\theta_u^T) \Big] dt\bigg|\ge \frac{\kappa}{8}u^2
\right)
\end{multline*}
Now by the Chebyshev inequality, \eqref{robust} and using bounds for
the moments of  stochastic integral (see e.g. \cite{LSI})
\begin{multline*}
\P_{\theta_0}\left( \bigg|\int_0^T
h^*\Big(\pi_t^{\theta_u^T}-\pi_t^{\theta_0}\Big)d\bar B_t\bigg|\ge
\frac{\kappa}{8}u^2
\right)\le\\
\shoveright
{
\left(\frac{8}{u^2 \kappa}\right)^{2m} \E_{\theta_0}
\bigg|\int_0^T h^*\Big(\pi_t^{\theta_u^T}-\pi_t^{\theta_0}\Big)d\bar B_t\bigg|^{2m} \le
}
\\
\shoveright
{
\left(\frac{8}{u^2 \kappa}\right)^{2m} \big(m(2m-1)\big)^m
T^{m-1}\|h\|^{2m}\int_0^T\E_{\theta_0}
\big\|\pi_t^{\theta_u^T}-\pi_t^{\theta_0}\big\|^{2m}dt
\le
} \\
\left(\frac{8}{u^2 \kappa}\right)^{2m} \big(m(2m-1)\big)^m
T^{m-1}\|h\|^{2m} T C_4 \frac{u^{2m}}{T^m} := \frac{C_5}{u^{2m}}.
\end{multline*}
Using the estimate \eqref{key},
\begin{multline*}
\P_{\theta_0}\left(\bigg|\frac 1 2\int_0^T
\Big[\Big(h^*\pi_t^{\theta_u^T}-h^*\pi_t^{\theta_0}\Big)^2
-g(\theta_0,\theta_u^T) \Big] dt\bigg|\ge \frac{\kappa}{8}u^2
\right) \le \\
 \left(\frac{4 T}{\kappa u^2 } \right)^{2m}\E_{\theta_0}
\left( \frac 1 T\int_0^T
\Big[\big(h^*\pi_t^{\theta_u^T}-h^*\pi_t^{\theta_0}\big)^2
-g(\theta_0,\theta_u^T) \Big] dt \right)^{2m}
\le
\end{multline*}
\begin{multline*}
\left(\frac{4 T}{\kappa u^2 } \right)^{2m}
C_6\left[\bigg( \frac{u^2}{T^{3/2}} \bigg)^{2m} +
\frac{1}{T^{2m}}\right]
\le \\
\left(\frac{4 T}{\kappa u^2 } \right)^{2m} C_7|\varTheta|^{2m}\left(
\frac{u}{T} \right)^{2m} = \frac{C_8}{u^{2m}},
\end{multline*}
where we used the fact $|u/\sqrt{T}|\le|\varTheta|$ (the diameter of
$\varTheta$). This verifies \eqref{instead} and thus the statement
of the lemma.
\end{proof}

\begin{lem}
Under assumptions of Theorem \ref{thm}, the finite dimensional
distributions of $Z_T(u)$ converge weakly to those of the process
$$
Z(u)=\exp\left(\sqrt{I(\theta_0)}u\zeta-\frac 1 2 I(\theta_0)
u^2\right), \quad u\in\Real,
$$
uniformly in $\theta_0\in\K$, where $\zeta$ is a standard Gaussian
random variable (i.e. \eqref{dva} of Theorem \ref{IH-thm} holds). In
particular, $$\hat u:=\argmax_{u\in \Real}Z(u)$$ is a zero mean
Gaussian random variable with variance $1/I(\theta_0)$ (i.e.
\eqref{tri} of Theorem \ref{IH-thm} holds as well).
\end{lem}
\begin{proof}
Recall the definition of the process $Z_T(u)$
$$
Z_T(u)= \exp\left\{\int_0^T
\big(h^*\pi^{\theta_u^T}_t-h^*\pi^{\theta_0}_t\big)d\bar B_t -\frac
1 2 \int_0^T \big(h^*\pi^{\theta_u^T}_t-h^*\pi^{\theta_0}_t\big)^2
dt\right\}.
$$
Using \eqref{Jacfla}, \eqref{Jac} and \eqref{ui}, similarly to the
proof of \eqref{gg}, we have
\begin{multline*}
\E_{\theta_0}\big|\pi^{\theta_0+\delta}_t-\pi_t^{\theta_0}-
\delta\dot \pi^{\theta_0}_t\big|
 =\\
\E_{\theta_0}\bigg|\int_0^t
 D\pi^{\theta_0}_{s,t}\big(\pi^{\theta_0+\delta}_s\big)\cdot \big(\Lambda^*(\theta_0+\delta)-\Lambda^*(\theta_0)\big)
\pi^{\theta_0+\delta}_{s}ds - \\
\delta \int_0^t
 D\pi^{\theta_0}_{s,t}\big(\pi^{\theta_0}_s\big)\cdot \Lambda'^*(\theta_0)
\pi^{\theta_0}_{s}ds\bigg| =  o(\delta^2)
\end{multline*}
uniformly in $t\ge 0$. Hence
$$
\int_0^T
\E_{\theta_0}\left(h^*\pi^{\theta_u^T}_t-h^*\pi_t^{\theta_0}-
\frac{u}{\sqrt{T}}h^*\dot \pi^{\theta_0}_t\right)^2dt
\xrightarrow{T\to\infty} 0,
$$
which implies
$$
\E_{\theta_0}\left|\int_0^T
\Big(h^*\pi^{\theta_u^T}_t-h^*\pi_t^{\theta_0}\Big)^2 dt -
\frac{u^2}{T}
\int_0^T\big(h^*\dot\pi^{\theta_0}_t\big)^2dt\right|\xrightarrow{T\to\infty}
0
$$
and in turn, by \eqref{a-3},
$$
\P_{\theta_0}-\lim_{T\to\infty}\int_0^T
\Big(h^*\pi^{\theta_u^T}_t-h^*\pi_t^{\theta_0}\Big)^2 dt =u^2
I(\theta_0),
$$
uniformly on compacts $\K\in \varTheta$. By the CLT for stochastic
integrals (see e.g. Proposition 1.20 in \cite{K04}),
$$
\int_0^T \big(h^*\pi^{\theta_u^T}_t-h^*\pi^{\theta_0}_t\big)d\bar
B_t
$$
converges weakly to a Gaussian random variable with zero mean and
variance $u^2I(\theta_0)$, uniformly on compacts $\theta_0\in \K$.
This implies the weak convergence of the one (and all finite) dimensional
distributions of $Z_T(u)$ to $Z(u)$. By \eqref{robust},
$I(\theta_0)$ is finite and  assuming that it is positive uniformly
on compacts in $\varTheta$, the maximizer of $Z(u)$ is unique and
equals $\hat u = \zeta/\sqrt{I(\theta_0)}$ as claimed.
\end{proof}

\section{An example} \label{sec-ex}
In this section we demonstrate with a simple example, how  the
conditions of Theorem \ref{thm} can be verified explicitly. Let
$S_t$ be a Markov chain with values in $\{0,1\}$, initial
distribution $\P(S_0=1)=\nu$ and transition matrix
$$
\Lambda = \theta
\begin{pmatrix}
-1 & \ \ 1\\
\ \ 1 & -1
\end{pmatrix},
$$
where $\theta$ is an unknown parameter, which controls the switching
rate of the chain. Suppose, it is known that the actual value of
this parameter $\theta_0$ lies within an interval
$\varTheta:=(\theta_{\min},\theta_{\max})\subset \Real_+$, $\theta_{\max}>\theta_{\min}>0$. The chain
is observed in the Gaussian white noise channel, i.e.
$$
X_t = \int_0^t S_r dr + B_t, \quad t\ge 0.
$$
The filtering process  $\pi^\theta_t=\P(S_t=1|\F^X_t)$ in this case
satisfies the SDE:
\begin{equation}\label{pipi}
d\pi^\theta_t = \theta(1-2\pi^\theta_t)dt +
\pi^\theta_t(1-\pi^\theta_t)\big(dX_t - \pi^\theta_tdt\big), \quad
t\in[0,T]
\end{equation}
started from $\pi^\theta_0=\nu$. The likelihood function is
$$
L_T(\theta;X^T) = \exp\left\{\int_0^T \pi^\theta_t
dX_t-\frac{1}{2}\int_0^T \big(\pi^\theta_t\big)^2 dt\right\}.
$$
The MLE of $\theta$ is found by maximizing $L_T(\theta;X^T)$  over
$\theta\in \bar\varTheta=[\theta_{\min}, \theta_{\max}]$. The
condition \eqref{a-1} is  satisfied and we should check  \eqref{a-2}
and \eqref{a-3}.

\subsection{The identifiability condition \eqref{a-2}}

Let $(\check \pi^{\theta_0}_t,\check \pi^\theta_t)$ be a stationary
(under $\P_{\theta_0}$) copy of the process defined by
\begin{align*}
 d\check\pi^{\theta_0}_t &= \theta_0(1-2\check\pi^{\theta_0}_t)dt +
\check\pi^{\theta_0}_t(1-\check\pi^{\theta_0}_t) d\bar B_t, \\
 d\check\pi^\theta_t & = \theta(1-2\check\pi^\theta_t)dt +
\check\pi^\theta_t(1-\check\pi^\theta_t)\big( d\bar B_t +(\check\pi^{\theta_0}_t- \check\pi^\theta_t)dt\big)
\end{align*}
where $d\bar B_t = dX_t - \check\pi^{\theta_0}_tdt$ is the corresponding innovation Brownian motion
with respect to $\F^X_t\vee \sigma\{\check\pi^{\theta_0}_0\}$.
Introduce an auxiliary process $q^\theta_t$, solving the equation
$$
dq^\theta_t = \theta(1-2q^\theta_t)dt +
q^\theta_t\big(1-q^\theta_t\big)d\bar B_t,
$$
subject to $q^\theta_0 = \check \pi^\theta_0$.
Heuristically, it is clear that if $|\check\pi^{\theta_0}_t- \check\pi^\theta_t|$ is small
on average for $|\theta- \theta_0|\ge r>0$ , then the distribution of $\check\pi^\theta_t$ should be close
to the distribution of $q^\theta_t$. But the latter satisfies an It\^o equation, corresponding
to the filtering problem for the signal with the switching rate $\theta$. This, in turn, would
imply that the signals with well separated $\theta$ and $\theta_0$ can be filtered with the same
steady state error. The latter can be argued false in our case and hence
$\check\pi^{\theta_0}_t$ and $\check\pi^\theta_t$  cannot be close. The rest is the
precise realization of this heuristics.

The difference
$\Delta_t:=\check \pi^\theta_t-q^\theta_t$ solves
$$
d\Delta_t = -2 \theta \Delta_t dt + \alpha_tdt + \Delta_t
\big(1-\check \pi_t^\theta-q^\theta_t\big)d\bar B_t,
\quad \Delta_0=0
$$
where $\alpha_t = \check \pi_t^\theta\big(1-\check
\pi_t^\theta\big)(\check \pi_t^{\theta_0}-\check \pi_t^\theta)$ and
hence $V_t = \E_{\theta_0}\Delta_t^2$ satisfies
$$
\dot V_t = -4 \theta V_t  + 2\E_{\theta_0} \Delta_t \alpha_t +
\E_{\theta_0} \Delta_t^2 \big(1-\check \pi_t^\theta-q^\theta_t\big)^2
\le C_1 V_t  + C_2 \sqrt{\E_{\theta_0} \alpha^2_t}.
$$
This implies
$$
\E_{\theta_0} \big(\check \pi^\theta_t-q^\theta_t\big)^2 \le
\frac{C_2}{C_1}\big(e^{C_1 t}-1\big) \sqrt{\E_{\theta_0}
\big(\check\pi^{\theta_0}_t-\check \pi^\theta_t\big)^2}
$$
and hence
\begin{multline}\label{vot-tak}
\E_{\theta_0} \big(\check \pi^{\theta_0}_t - q^\theta_t\big)^2 \le
2\E_{\theta_0} \big(\check \pi^{\theta_0}_t - \check
\pi^\theta_t\big)^2 +
2\E_{\theta_0} \big(\check \pi^\theta_t - q^\theta_t\big)^2 \le \\
2\sqrt{\E_{\theta_0} \big(\check \pi^{\theta_0}_t - \check
\pi^\theta_t\big)^2}+
2\frac{C_2}{C_1}\big(e^{C_1 t}-1\big) \sqrt{\E_{\theta_0} \big(\check\pi^{\theta_0}_t-\check \pi^\theta_t\big)^2}\le\\
\rho(t) \sqrt{\E_{\theta_0} \big(\check\pi^{\theta_0}_t-\check
\pi^\theta_t\big)^2}
\end{multline}
with $\rho(t):=2+2\frac{C_2}{C_1}\big(e^{C_1 t}-1\big)>0$ for any
$t>0$ (regardless of the sign of $C_1$). On the other hand, by the Jensen inequality
$$
\Big|\E_{\theta_0} \big(\check \pi^{\theta_0}_t\big)^2 -
\E_{\theta_0}\big(q^\theta_t\big)^2\Big|\le 2\E_{\theta_0}\big|
\check \pi^{\theta_0}_t - q^\theta_t\big|\le 2
\sqrt{\E_{\theta_0}\big( \check \pi^{\theta_0}_t -
q^\theta_t\big)^2}.
$$
The distribution of $q^\theta_t$ converges weakly to the stationary
distribution of  $\pi^{\theta}_t$  under $\P_\theta$ (and not $\P_{\theta_0}$!), i.e. to the
distribution of $\check\pi^\theta_t$ under $\P_\theta$. Thus for any
$\eps>0$ we may choose $T(\eps)$ large enough such that
$$
\big|\E_{\theta_0}\big(q^\theta_t\big)^2 -
\E_\theta\big(\check\pi^\theta_t\big)^2\big|\le \eps, \quad \forall
t\ge T(\eps).
$$
The distributions of $\check \pi^{\theta_0}_t$ (under
$\P_{\theta_0}$) and of $\check \pi^{\theta}_t$ (under
$\P_{\theta}$) can be found explicitly by solving the corresponding
Kolmogorov equations (see e.g. Section 15.4, \cite{LSII}) and
$\E_{\theta_0}\big(\check \pi^{\theta_0}_t\big)^2\ne
\E_\theta\big(\check \pi^\theta_t\big)^2$ whenever $\theta\ne
\theta_0$, is checked by direct calculation. Moreover,
$\E_\theta\big(\check \pi^\theta_t\big)^2$ is easily seen to be
continuous in $\theta$ on $[\theta_{\min},\theta_{\max}]$, and thus
$$
g(r):=\inf_{{\theta_0}\in \K}\inf_{\theta:|\theta-{\theta_0}|\ge
r}\big|\E_{\theta_0}\big(\check\pi^{\theta_0}_t\big)^2 -
\E_\theta\big(\check\pi^\theta_t\big)^2\big|>0.
$$
But then by \eqref{vot-tak}, for any $|{\theta_0}-\theta|\ge r>0$
\begin{multline*}
\E_{\theta_0} \big(\check \pi^{\theta_0}_t - \check
\pi^\theta_t\big)^2 \ge\frac{ \Big|\E_{\theta_0} \big(\check
\pi^{\theta_0}_t\big)^2 -
\E_{\theta_0}\big(q^\theta_t\big)^2\Big|^4} {16 \rho^2(t)}
 \ge \\
\frac{ 1/8\Big|\E_{\theta_0} \big(\check \pi^{\theta_0}_t\big)^2 -
\E_\theta\big(\check \pi^\theta_t\big)^2\Big|^4- \eps^4}
{16\rho^2\big(T(\eps)\big)} \ge \frac{ \frac{1}{8}g^4(r)- \eps^4}
{16 \rho^2\big(T(\eps)\big)}.
\end{multline*}
The required property  \eqref{ident} now follows by arbitrariness of
$\eps$ and positiveness of $\rho(t)$, $t\ge 0$.

\begin{rem}
The coupling argument, used in this example, is applicable in the
general $d>2$ case, namely $\E_{\theta_0}\big(h^*\check\pi^{\theta_0}_t-h^*\check\pi^\theta_t\big)^2$
can be similarly shown to be lower bounded by a quantity, proportional to
$\big|\E_{\theta_0}f\big(h^*\check\pi^{\theta_0}_t\big)-\E_{\theta}f\big(h^*\check\pi^\theta_t\big)\big|$
with $f(x)=x^2$ or $f(x)=x$, etc. The latter means that the stationary laws of two $d$-dimensional
diffusions are to be studied, instead of the law of $2d$-dimensional diffusion $(\pi^{\theta_0}_t,\check\pi^\theta_t)$.
In particular the identifiability follows, if one is able to show that the
laws of $h^*\check \pi^{\theta_0}_t$ under $\P_{\theta_0}$ and $h^*\check
\pi^\theta_t$ under $\P_\theta$ have  different moments, uniformly for separated
$\theta$ and ${\theta_0}$. In the example, this was
possible due to explicit expression available for the probability
density of the filtering process when $d=2$.
\end{rem}
\subsection{The regularity condition \eqref{a-3}}
The derivative $\dot \pi^{\theta_0}_t$ satisfies the equation
\begin{multline}
\label{der} d\dot \pi^{\theta_0}_t = (1-2\pi^{\theta_0}_t)dt -
\Big(2\theta_0+\pi^{\theta_0}_t(1-\pi^{\theta_0}_t)\Big)\dot
\pi^{\theta_0}_tdt +
\\
(1-2\pi^{\theta_0}_t)\dot \pi^{\theta_0}_td\bar B_t, \quad \dot
\pi^{\theta_0}_0=0.
\end{multline}
and hence the pair $( \pi^{\theta_0}_t, \dot \pi^{\theta_0}_t)$ is a
Markov-Feller process. The formula \eqref{Jacfla} yields
$$
\dot\pi^{\theta_0}_t =-2 \int_0^t
D\pi^{\theta_0}_{s,t}\big(\pi^{\theta_0}_s\big)
\pi^{\theta_0}_{s}ds,
$$
and, in turn,  the bound \eqref{Jac} implies
\begin{equation}\label{bsup}
\sup_{t\ge 0}\E_{\theta_0}\big(\dot\pi^{\theta_0}_t\big)^2<\infty.
\end{equation}
This guarantees  existence of at least one invariant measure for $(
\pi^{\theta_0}_t, \dot \pi^{\theta_0}_t)$ (e.g. Theorem 2.1, Ch III,
\cite{Khas}). The uniqueness of this measure as well as the limits,
required in \eqref{a-3}, are deduced by standard arguments from
\eqref{bsup} and the fact that the distance between any  two
solutions of \eqref{der}, started from distinct initial conditions,
converges to zero with positive asymptotic exponential rate. Let $(
\check\pi^{\theta_0}_t, \dot {\check\pi}^{\theta_0}_t)$ be the
stationary pair, then \eqref{der} implies
\begin{multline*}
 \dot {\check \pi}^{\theta_0}_t = \dot {\check \pi}^{\theta_0}_0 e^{-2\theta_0 t}
+\int_0^t e^{-2\theta_0(t-s)}(1-2\check \pi^{\theta_0}_s)ds
-\\
 \int_0^t e^{-2\theta_0(t-s)}\check\pi^{\theta_0}_s(1-\check\pi^{\theta_0}_s)\dot {\check\pi}^{\theta_0}_s ds
+ \int_0^t e^{-2\theta_0(t-s)}
(1-2\check\pi^{\theta_0}_s)\dot{\check \pi}^{\theta_0}_sd\bar B_s,
\end{multline*}
and hence
\begin{multline*}
\E_{\theta_0}\left(\int_0^t e^{-2\theta_0(t-s)}(1-2\check
\pi^{\theta_0}_s)ds\right)^2\le
 4 (1+  e^{-4\theta_0 t})I(\theta_0)
 +\\
\frac{4}{2\theta_0}\int_0^t
e^{-2\theta_0(t-s)}\frac{1}{16}I(\theta_0) ds + 4 \int_0^t
e^{-4\theta_0(t-s)} I(\theta_0)ds\le C_1 I(\theta_0),
\end{multline*}
with a constant $C_1>0$. The process $\zeta_t := \int_0^t
e^{-2\theta_0(t-s)}(1-2\check \pi^{\theta_0}_s)ds$ is the solution
of the equation
$$
\dot{\zeta}_t = -2\theta_0 \zeta_t + (1-2\check \pi^{\theta_0}_t),
\quad \zeta_0=0.
$$
Elementary calculations show that
$$
\lim_{t\to\infty}\E_{\theta_0}(\zeta_t)^2 =
\frac{2\E_{\theta_0}\big(\check
\pi^{\theta_0}_0-1/2\big)^2}{4\theta^2_0},
$$
which is positive for any positive $\theta_0$. Hence
$\inf_{\theta_0\in \bar \varTheta}I(\theta_0)>0$  as required in
\eqref{a-3}.

\section{A discussion}\label{sec-dis}

The  result stated in Theorem \ref{thm} is extendable to the
vector parameter space $\varTheta$, since the key properties such as
\eqref{pr}, \eqref{Jac} and \eqref{Jacfla} do not depend on the
dimension of $\theta$. On the other hand, the setting where $h$
depends on the parameter, seems to be more delicate  and would
require additional effort, mainly because the  formula analogous to
\eqref{Jacfla} in this case is more intricate and involves
Skorokhod anticipating integrals (see Proposition 4.1 in
\cite{ChVH}). As was mentioned before, the requirement \eqref{a-1}
is essential and it is not obvious whether the claimed results hold
under weaker form of ergodicity of the chain $S$ (especially the
convergence of moments). The requirements \eqref{a-2} and
\eqref{a-3} seem to be quite natural, though it is not clear at what
level of generality they can be verified in terms of the model data.

\appendix
\section{An LLN for processes with short correlation}\label{sec-app}
The following is a version of Lemma 2.1 in \cite{Kh66}, adapted to
our purposes. The proof mostly follows the lines of the original proof.
\begin{lem}\label{khas-lem}
Let $\Phi(t;b)$ (we shall also write $\Phi(t)$ for brevity) be a stochastic
process with real values, depending on a  parameter $b>0$, such that
$\E \Phi(t;b)=0$ and
\begin{equation}
\label{mom} \E \big|\Phi(t;b)\big|^m \le Cb^m, \quad m=1,2,...
\end{equation}
for a  constant $C$, possibly depending on
$m$. Let $0\le t_1\le t_2\le ...\le t_n$, and assume that for any
$i\in \{1,...,n\}$
\begin{equation}
\label{kakmix} \left|\E \Phi(t_1)...\Phi(t_n) -
\E\Phi(t_1)...\Phi(t_i)\E \Phi(t_{i+1})...\Phi(t_n)\right|\le C_n
b^{n}\alpha(t_{i+1}-t_i)
\end{equation}
with a nonnegative decreasing function $\alpha(\tau)$, such that
$$
A_n := \int_0^\infty \tau^{n-1}\alpha(\tau)d\tau <\infty,\quad
n=1,...,k.
$$
Then
$$
\int_0^T...\int_0^T \big|\E
\Phi(s_1)...\Phi(s_{2k})\big|ds_1...ds_{2k}\le c_{2k} b^{2k}T^k,
$$
where $c_{2k}$ are constants, depending only on $k$ and
$A_1,...,A_k$.
\end{lem}

\begin{proof}
The lemma is proved by induction. The bound \eqref{kakmix} implies
$$
\int_0^T \int_0^T \big|\E \Phi(s_1)\Phi(s_2)\big|ds_1ds_2 \le 2C_2
b^2\int_0^T\int_0^t   \alpha(s)dsdt \le 2C_2 b^2 A_1 T,
$$
and hence the claim holds for $k=1$. Suppose now that the lemma has
been proved for $k\le n$. Let $s=(s_1,...,s_{2n+2})\in
[0,T]^{2n+2}$. Let $j_1,...,j_{2n+2}$ be the permutation of the
indices such that $ s_{j_1}\le s_{j_2}\le ...\le  s_{j_{2n+2}} $ and
let $r=r(s)$ be any index for which
$$
\max_{\ell=1,...,n}\big(s_{j_{2\ell+1}}-s_{j_{2\ell}}\big)=s_{j_{2r+1}}-s_{j_{2r}}.
$$
From \eqref{kakmix} it follows that
\begin{multline}\label{2.5}
\Big| \E \Phi(s_1)...\Phi(s_{2n+2})-\E
\Phi(s_1)...\Phi(s_{j_{2r}})\E
\Phi(s_{j_{2r+1}})...\Phi(s_{j_{2n+2}}) \Big|\le\\ C
b^{2n+2}\alpha(s_{j_{2r+1}}-s_{j_{2r}})
\end{multline}
Since $\Phi(t)$ is zero mean, \eqref{kakmix} implies
\begin{align*}
&\big|\E \Phi(s_1)...\Phi(s_{2n+2}) \big| \le C b^{2n+2}\alpha(s_{j_{2n+2}}-s_{j_{2n+1}}) \\
&\big|\E \Phi(s_{j_{2r+1}})...\Phi(s_{j_{2n+2}})\big|\le C
b^{2n+2-2r}\alpha(s_{j_{2n+2}}-s_{j_{2n+1}}).
\end{align*}
Moreover, by the H\"older inequality and \eqref{mom}, $ \E\big|
\Phi(s_1)...\Phi(s_{j_{2r}})\big| \le C b^{2r}, $ and thus
(hereafter $c$ is a constant, possibly depending on $n$, whose value
may differ from line to line)
\begin{multline*}
\Big| \E \Phi(s_1)...\Phi(s_{2n+2})-\E
\Phi(s_1)...\Phi(s_{j_{2r}})\E
\Phi(s_{j_{2r+1}})...\Phi(s_{j_{2n+2}})
\Big|\le \\
c  b^{2n+2}\alpha(s_{j_{2n+2}}-s_{j_{2n+1}}).
\end{multline*}
As $\alpha(\tau)$ decrease, the latter and \eqref{2.5} imply
\begin{multline}\label{2.6}
\Big| \E \Phi(s_1)...\Phi(s_{2n+2})-\E
\Phi(s_1)...\Phi(s_{j_{2r}})\E
\Phi(s_{j_{2r+1}})...\Phi(s_{j_{2n+2}})
\Big|\le \\
2c b^{2n+2}\alpha\big(\max(s_{j_{2n+2}}-s_{j_{2n+1}},
s_{j_{2r+1}}-s_{j_{2r}})\big)
\end{multline}
By the definition of $r$
\begin{multline*}
\sigma(s):= \max(s_{j_{2n+2}}-s_{j_{2n+1}}, s_{j_{2r+1}}-s_{j_{2r}})> \\
\frac{1}{n+1}\Big(\sum_{\ell=1}^n(s_{j_{2\ell+1}}-s_{j_{2\ell}})+(s_{j_{2n+2}}-s_{j_{2n+1}})\Big),
\end{multline*}
and as $\alpha(\tau)$ decreases, we have
\begin{multline*}
\int_0^T ...\int_0^T \alpha\big(\sigma(s)\big)ds\le \\ (2n+2)!
\idotsint\limits_{0<s_1<...<s_{2n+2}<T}
 \alpha
\Big\{
\frac{1}{n+1}\Big(\sum_{\ell=1}^n(s_{j_{2\ell+1}}-s_{j_{2\ell}})+(s_{j_{2n+2}}-s_{j_{2n+1}})\Big)
\Big\}ds.
\end{multline*}
Using the formula
$$
\int_0^\infty ...\int_0^\infty
\alpha(t_1+...+t_{n+1})dt_1...dt_{n+1}=\frac{1}{n!}\int_0^\infty u^n
\alpha(u)du,
$$
the following estimate is obtained:
\begin{equation}\label{2.9}
\int_0^T ...\int_0^T \alpha\big(\sigma(s)\big)ds\le c_n A_{n+1}
T^{n+1},
\end{equation}
where $c_n$ depends only on $n$. Further
\begin{multline*}
\int_0^T ...\int_0^T \big|\E \Phi(s_{j_1})...\Phi(s_{j_{2r}})\big|\big|\E \Phi(s_{j_{2r+1}})...\Phi(s_{j_{2n+2}})\big|ds\le \\
\sum_{\ell=1}^n
\int_0^T ...\int_0^T \big|\E \Phi(s_{j_1})...\Phi(s_{j_{2\ell}})\big|\big|\E \Phi(s_{j_{2\ell+1}})...\Phi(s_{j_{2n+2}})\big|ds\le\\
 (2n+2)! \sum_{\ell=1}^n
\int_0^T ...\int_0^T \big|\E \Phi(s_{1})...\Phi(s_{2\ell})\big|ds_1...ds_{2\ell}\times \\
\int_0^T ...\int_0^T\big|\E
\Phi(s_{2\ell+1})...\Phi(s_{2n+2})\big|ds_{2\ell+1}ds_{2n+2}
\end{multline*}
By the induction hypothesis, each of the terms in the sum on the
right hand side are bounded by $c_{2n} b^{2n+2} T^{n+1}$ and hence
using \eqref{2.6}, \eqref{2.9} we obtain
$$
\int_0^T...\int_0^T \big|\E
\Phi(s_1)...\Phi(s_{2n+2})\big|ds_1...ds_{2n+2} \le c_{2n+2}
b^{2n+2}T^{n+1}.
$$
\end{proof}

% ----------------------------------------------------------------

\end{document}